\documentclass[12pt]{amsart}
\usepackage{amscd,amsmath,amssymb,amsthm,amsfonts,epsfig,graphics}
\usepackage{geometry}
\usepackage{graphicx}
\usepackage{mathrsfs,amssymb}
\usepackage{bm}
\usepackage{url}
\usepackage{hyperref}

\theoremstyle{plain}

\newtheorem{corollary}{Corollary}

\newtheorem{lemma}{Lemma}

\newtheorem{proposition}{Proposition}

\newtheorem{theorem}{Theorem}
\numberwithin{equation}{section}

\begin{document}

\title[An Asymptotic formula for the zeros]{An asymptotic formula for the zeros of the deformed exponential function}

\author{Cheng Zhang}

\address{Department of Mathematics, Johns Hopkins University, Baltimore, MD21218, United States}

\email{czhang67@jhu.edu}









\begin{abstract}
       We study the asymptotic representation for the zeros of the deformed exponential function $\sum\nolimits_{n = 0}^\infty  {\frac1{n!}{q^{n(n - 1)/2}{x^n}}} $, $q\in (0,1)$. Indeed, we obtain an asymptotic formula for these zeros: \[x_n=- nq^{1-n}(1 + g(q)n^{-2}+o(n^{-2})),n\ge1,\]
        where $g(q)=\sum\nolimits_{k = 1}^\infty  {\sigma (k){q^k}}$ is the generating function of the sum-of-divisors function $\sigma(k)$. This improves earlier results by Langley\cite{langley2000certain} and Liu\cite{liu1998some}. The proof of this formula is reduced to estimating the sum of an alternating series, where the Jacobi's triple product identity plays a key role.
\end{abstract}

\maketitle

\section*{Introduction}

In this paper, we are interested in the deformed exponential function\begin{equation}\label{fx}f(x) = \sum\limits_{n = 0}^\infty  {\frac{{{x^n}}}{{n!}}{q ^{ n(n - 1)/2}}},\  0<q<1,
\end{equation}
which is the  unique entire solution to the Cauchy problem of the functional differential equation
\[y'(x) = y(qx), 0<q<1, y(0) = 1.\]
The function $f(x)$ appears naturally and frequently in pure mathematics as well as statistical physics. There are still many open questions and conjectures on it, especially on the asymptotic formula for its zeros. Indeed, it relates closely to the generating function for Tutte polynomials of the complete graph $K_n$ in combinatorics, the Whittaker and Goncharov constants in complex analysis, and the partition function of one-site lattice gas with fugacity $x$ and two-particle Boltzmann weight $q$ in statistical mechanics\cite{scott2005repulsive}. One may see more interesting applications and conjectures concerning this function in Liu\cite{liu1998some} and Sokal\cite{sokal2009}.

In 1972 Morris et al.\cite{morris1972phragmen} used a theorem of Laguerre to prove that $f(x)$ has infinitely many real zeros and these zeros are all negative and simple. They also obtained that there is no other zero to the analytic extension (to the complex plane) of $f(x)$ by using the so-called multiplier sequence (a modest gap in their proof was filled by Iserles\cite{iserles1993generalized}). In \cite{morris1972phragmen} and \cite{iserles1993generalized}, the authors put forward several conjectures on the zeros of $f(x)$. They have been investigated by several authors(see \cite{liu1998some}\cite{langley2000certain}\cite{sokal2012leading}). Recently, more interesting conjectures on the asymptotic formula for the zeros were introduced by Sokal\cite{sokal2009}.  In \cite{langley2000certain}, the author was able to show that the zeros form a strictly decreasing sequence of negative numbers $(x_n)$, $n\ge 1$, such that
\begin{equation}\label{ratio}\frac{x_{n+1}}{x_{n}}=\frac1q\Big(1+\frac1n\Big)+o(n^{-2}),\ n\ge 2,\end{equation}and
\begin{equation}\label{xnlangley}x_{n}=-nq^{1-n}(\gamma+o(1)),\ n\ge1,\end{equation}
where $\gamma$ is some positive constant independent of $n$.

In this paper, we use the power series representation of $f(x)$ to show that the constant $\gamma=1$ in \eqref{xnlangley}. The well-known Jacobi's triple product identity plays a key role in the proof. Meanwhile, we obtain that the error term in \eqref{xnlangley} is exactly $O(n^{-2})$, whose leading coefficient, as a function of $q$, is the generating function of the sum-of-divisors function in number theory. It is not difficult to see that the signs of the extrema of $f(x)$ are alternating, namely ${( - 1)^n}f({x_n}/q) > 0,\ n\ge 1$, which means that the function oscillates on the negative real axis. So we prove a corollary that describes the asymptotic and oscillatory behavior of the function $f(x)$ on the negative real axis, where the infinite product representation of $f(x)$ is used.
\begin{theorem}\label{thm}The zeros of $f(x)$ constitute one strictly decreasing sequence of negative numbers $(x_n)$, $n\ge 1$, such that
\[{x_n} =  - nq^{1-n}(1 + g(q)n^{-2}+o(n^{-2})),\]
where $g(q)$ is the generating function of $\sigma(k)$, namely
\[g(q) = \sum\limits_{k = 1}^\infty  {\sigma (k){q^k}}.\]
Here $\sigma(k)$ is the sum of all the positive divisors of the positive integer $k$.
\end{theorem}

\begin{corollary}\label{coro} The oscillation amplitude $A_n$=$|f({x_n}/q)|$ has an asymptotic formula
\[A_n= Cn^{-3/2}e^nq ^{-n(n+1)/2},\]
where $C=C(n, q)$ satisfies $C_1(q)< C< C_2(q)$. Here $C_1$ and $C_2$ are some positive numbers independent of n.
\end{corollary}

\section{Proof of the results}
For simplicity, we set $\alpha=1/q>1$ in the proof.

To obtain our theorem, we study the value of $f( - (k + \lambda {k^{ - 1}}){\alpha ^{k - 1}}), \lambda>0$, $k\ge 1$, which can be written as an alternating series
\[f( - (k + \lambda {k^{ - 1}}){\alpha ^{k - 1}})=  \sum\limits_{n = 0}^\infty  {{( - 1)}^n{u_n}},\]
where \[u_n={\frac{{{{(k + \lambda {k^{ - 1}})}^n}}}
{{n!}}} {\alpha ^{n(2k - n- 1)/2}}.\]
Denote
\[{v_j} = {u_{2k - j - 1}} - {u_j}\quad (0 \le j \le k - 1).\]
At first, we need to figure out some useful properties of $v_j$.

\begin{lemma}\label{lemmavj} For any integer $k \ge 1$, we have
\[\quad {v_j} > 0 \quad (0 \le j \le k - 1),\]
and there exists a positive integer $N=N(\alpha, \lambda)$ such that for large $k$
\[{v_j} < {v_{j + 1}} \quad (0 \le j \le k - N).\]
\end{lemma}

\begin{proof}
Note that
\[\prod\limits_{i = 1}^{2k - 1 - 2j} {(j + i)}  < {k^{2k - 1 - 2j}} < {(k + \lambda {k^{ - 1}})^{2k - 1 - j}}\quad (0 \le j \le k - 1).\]
This implies
\[\frac{{{{(k + \lambda {k^{ - 1}})}^j}}}
{{j!}} < \frac{{{{(k + \lambda {k^{ - 1}})}^{2k - 1 - j}}}}
{{(2k - 1 - j)!}}\quad (0 \le j \le k - 1).\]
So
\[{\rm{ }}{u_j} < {u_{2k - 1 - j}}\quad (0 \le j \le k - 1),\]
which gives the first inequality.

Let \[{w_j} = \frac{{(2k - 2 - j)!}}
{{(j + 1)!{{(k + \lambda {k^{ - 1}})}^{2k - 3 - 2j}}}}.\]
Then we have
\[0<w_j<1\]
and
\[\frac{{{w_{j + 1}}}}
{{{w_j}}} = \frac{{{{(k + \lambda {k^{ - 1}})}^2}}}
{{(j + 2)(2k - 2 - j)}} > 1.\]
For some positive integer $N$, we denote
\[t = 2k - 1 - j\quad (k + N-1 \le t \le 2k - 1)\]
and
\[g(t) = \frac{1}
{t} - \frac{{(2k - t){w_{k - N}}}}
{{{{(k + \lambda {k^{ - 1}})}^2}}} - {\alpha ^{t - k}}\left( {\frac{{1 - {w_{k - N}}}}
{{k + \lambda {k^{ - 1}}}}} \right).\]
Direct calculation yields
\[g'(t) =  - \frac{1}
{{{t^2}}} + \frac{{{w_{k - N}}}}
{{{{(k + \lambda {k^{ - 1}})}^2}}} - \left( {\frac{{1 - {w_{k - N}}}}
{{k + \lambda {k^{ - 1}}}}} \right){\alpha ^{t - k}}\ln \alpha\]
and
\[g''(t) = \frac{2}
{{{t^3}}} - \left( {\frac{{1 - {w_{k - N}}}}
{{k + \lambda {k^{ - 1}}}}} \right){\alpha ^{t - k}}{(\ln \alpha )^2}.\]
Since $0<w_j<1$, $g''(t)$ is decreasing for $t>0$.

When $N$ is sufficiently large($N>N_1(\alpha, \lambda)$),
\[g''(k+N-1) = \left[ {2 - C{\alpha ^{N - 1}}{{(\ln \alpha )}^2}} \right]{k^{ - 3}} + O({k^{ - 4}})<0,\]
and
\[g'(k + N - 1) = \left[ {2N - 2 - C{\alpha ^{N - 1}}\ln \alpha } \right]{k^{ - 3}} + O({k^{ - 4}}) < 0\] hold for large $k$,
where \[C = \lambda (2N - 3) + \frac{1}{6}(N - 2)(N - 1)(2N - 3).\]
So $g'(t)$ and $g(t)$ are also decreasing for $t\ge k+N-1$.

When $N$ is large enough($N>N_2(\alpha, \lambda)$),
\[g(k + N - 1) = \left[ {\lambda (2N - 1) + \frac{N}
{6}(N - 1)(2N - 1) - C{\alpha ^{N - 1}}} \right]{k^{ - 3}} + O({k^{ - 4}})<0\]holds for large $k$.
So if $N>{\rm max}\{N_1, N_2\}$, we obtain for large $k$
\[g(t) \le g(k + N-1) < 0, \quad t\ge k+N-1.\]
Therefore, \[\frac{1}
{{2k - 1 - j}} - \frac{{(j + 1){w_{k - N}}}}
{{{{(k + \lambda {k^{ - 1}})}^2}}} - {\alpha ^{k - j - 1}}\left( {\frac{{1 - {w_{k - N}}}}
{{k + \lambda {k^{ - 1}}}}} \right) < 0\quad (0 \le j \le k - N).\]
This implies \[\left( {{\alpha ^{k - j - 1}} - \frac{{j + 1}}
{{k + \lambda {k^{ - 1}}}}} \right){w_{k - N}} < {\alpha ^{k - j - 1}} - \frac{{k + \lambda {k^{ - 1}}}}
{{2k - 1 - j}}\quad (0 \le j \le k - N).\]
Since $w_j<w_{j+1}$, we have \[\left( {{\alpha ^{k - j - 1}} - \frac{{j + 1}}
{{k + \lambda {k^{ - 1}}}}} \right){w_j} < {\alpha ^{k - j - 1}} - \frac{{k + \lambda {k^{ - 1}}}}
{{2k - 1 - j}}\quad (0 \le j \le k - N).\]
It is equivalent to \[{v_j} < {v_{j + 1}}\quad (0 \le j \le k - N),\]which completes the proof.
\end{proof}

To estimate the sum of the alternating series $\sum\nolimits_{n=0}^\infty  {(-1)^nu_n}$, we need to use the Jacobi's triple product identity. Let \[h(\alpha ){\text{  = }}\sum\limits_{j = 1}^\infty  {(2j - 1){{( - 1)}^{j - 1}}{\alpha ^{ - j(j - 1)/2}}}\]
and
\[H(\alpha ) = \sum\limits_{j = 1}^\infty  {\frac{j}
{6}(j - 1)(2j - 1){{( - 1)}^j}{\alpha ^{ - j(j - 1)/2}}}.\]
We denote their partial sums by $h_n(\alpha)$ and $H_n(\alpha)$, namely
\[h_n(\alpha ){\text{  = }}\sum\limits_{j = 1}^n  {(2j - 1){{( - 1)}^{j - 1}}{\alpha ^{ - j(j - 1)/2}}}\]
and
\[H_n(\alpha ) = \sum\limits_{j = 1}^n  {\frac{j}
{6}(j - 1)(2j - 1){{( - 1)}^j}{\alpha ^{ - j(j - 1)/2}}}.\]

\begin{proposition}\label{hardy}(Jacobi's triple product identity)
\[h(\alpha)= \prod\limits_{k = 1}^\infty  {{{(1 - {\alpha ^{ - k}})}^3}}.\]
\end{proposition}
\begin{proof}See Hardy\cite{hardy1979introduction}, Chapter 3.
\end{proof}

The following lemma is a simple property of the partial sums.
\begin{lemma}\label{lemmapsum}For any fixed $\alpha>1$, we have $h(\alpha)>0$ and $H(\alpha)>0$. Further, there exists a positive integer $N=N(\alpha)$ such that for any integer $m>N$,\[ \frac{{{H_{2m - 1}}(\alpha )}}
{{{h_{2m - 1}}(\alpha )}} < \frac{{{H_{2m+1}}(\alpha )}}
{{{h_{2m + 1}}(\alpha )}} < \frac{{{H_{2m}}(\alpha )}}
{{{h_{2m}}(\alpha )}} < \frac{{{H_{2m - 2}}(\alpha )}}
{{{h_{2m - 2}}(\alpha )}}.\]
\end{lemma}

\begin{proof}By the Jacobi's triple product identity above, we have $h(\alpha)>0$ and
\[H(\alpha ) = \frac{\alpha }
{3}h'(\alpha ) = h(\alpha )\sum\limits_{k = 1}^\infty  {\frac{k}
{{{\alpha ^k} - 1}}}  > 0.\]
Hence, there exists a positive integer $N=N(\alpha)$ such that for any integer $m>N$, both $h_{2m-2}(\alpha)$ and $H_{2m-1}(\alpha)$ are positive. Since $0 < {h_{2m - 2}}(\alpha ) < {h_{2m}}(\alpha ) < {h_{2m + 1}}(\alpha ) < {h_{2m - 1}}(\alpha )$
and $0 < {H_{2m - 1}}(\alpha ) < {H_{2m + 1}}(\alpha ) < {H_{2m}}(\alpha ) < {H_{2m - 2}}(\alpha )$
, we have\[ \frac{{{H_{2m - 1}}(\alpha )}}
{{{h_{2m - 1}}(\alpha )}} < \frac{{{H_{2m + 1}}(\alpha )}}
{{{h_{2m+ 1}}(\alpha )}} < \frac{{{H_{2m}}(\alpha )}}
{{{h_{2m}}(\alpha )}} < \frac{{{H_{2m - 2}}(\alpha )}}
{{{h_{2m - 2}}(\alpha )}}.\]
\end{proof}

\begin{lemma}\label{sign} Suppose $\lambda h(\alpha)-H(\alpha)\ne 0$. Let \[{\xi _n}(\lambda ) =  - (n + \lambda /n){\alpha ^{n - 1}}.\]
Then for large $n$,
\[{( - 1)^n}[\lambda h(\alpha)-H(\alpha)]f( \xi_n) > 0.\]
\end{lemma}

\begin{proof}Fix an integer $j\ge 0$. Then direct calculation shows, for large $k$
\[{v_{k - j}} = \frac{{{{(k + \lambda {k^{ - 1}})}^{k + j - 3}}}}
{{(k + j - 1)!}}\left[ {(2j - 1)\lambda  + \frac{j}
{6}(j - 1)(2j - 1) + O({k^{ - 1}})} \right]{\alpha ^{(k + j - 1)(k - j)/2}}.\]
By Lemma \ref{lemmavj} there exists a positive integer $N_1=N_1(\alpha, \lambda)$ such that
\[ v_{k-N_1}>v_{k-N_1-1}>\cdot\cdot\cdot>v_{0}.\]
If $\lambda h(\alpha)-H(\alpha) <0$, by Lemma \ref{lemmapsum} there exists a positive integer $N_2(\alpha, \lambda)$ such that  for any integers $m\ge N_2$,  $\lambda h_{2m-1}-H_{2m-1}<0$.

Let $N={\rm max} \{ [N_1/2], N_2\}$.
Direct calculation yields
\[\sum\limits_{j = 1}^{2N - 1} {{v_{k - j}}{{( - 1)}^{j - 1}}}  = \frac{{{{(k + \lambda {k^{ - 1}})}^{k - 2}}}}
{{k!}}\left[ {\lambda {h_{2N - 1}}(\alpha ) - {H_{2N - 1}}(\alpha ) + O({k^{ - 1}})} \right]{\alpha ^{k(k - 1)/2}}.\]
Note that for large $k(k>K(\alpha, \lambda))$,
\[\lambda {h_{2N - 1}} - {H_{2N - 1}} + O(k^{ - 1})<0.\]
Hence for each $k>{\rm max}\{2N+1, K\}$,
\[ \sum\limits_{j = 1}^{2N - 1} {{v_{k - j}}{{( - 1)}^{j - 1}}}<0\]
and
\[\sum\limits_{j = 2N}^k {{v_{k - j}}{{( - 1)}^{j - 1}} < {v_{k - 2N - 1}} - {v_{k - 2N}} < 0}.\]
So we have
\[{( - 1)^k}\sum\limits_{n = 0}^{2k - 1} {{u_n}{{( - 1)}^n}}  = \sum\limits_{j = 1}^k {{v_{k - j}}{{( - 1)}^{j - 1}} < 0}.\]
Since \[\left| {\sum\limits_{n = 0}^{2k - 1} {{u_n}{{( - 1)}^n}} } \right| > {C_1}\frac{{{{(k + \lambda {k^{ - 1}})}^{k - 2}}}}
{{k!}}{\alpha ^{k(k - 1)/2}} > \frac{{{{(k + \lambda {k^{ - 1}})}^{2k}}}}
{{(2k)!}}{\alpha ^{ - k}} = {u_{2k}},\]
 and \[u_{2k}>\left| {\sum\limits_{n = 2k}^\infty  {{u_n}{{( - 1)}^n}} } \right|,\]
we derive \[{( - 1)^k}f( - (k + \lambda {k^{ - 1}}){\alpha ^{k - 1}}) ={( - 1)^k}\sum\limits_{n = 0}^\infty  {{u_n}{{( - 1)}^n}}  < 0\] for large $k$.
Similarly, if  $\lambda h(\alpha)-H(\alpha)>0$, we also have for large $k$
\[{( - 1)^k}f( - (k + \lambda {k^{ - 1}}){\alpha ^{k - 1}}) > 0,\]which completes the proof.
\end{proof}

\begin{proposition}\label{langley}For any integer $n\ge 1$, we have \[x_{n+1}<\alpha x_n<0.\]
For large $n$, \[\frac{{{x_{n + 1}}}}
{{\alpha{x_n}}} = 1 + \frac{1}
{n}+o({n^{ - 2}}).\]
\end{proposition}
\begin{proof}See Langley\cite{langley2000certain}.
\end{proof}

\begin{proof}[Proof of Theorem \ref{thm}] Set ${\lambda _0} = 1 + H(\alpha )/h(\alpha )$ and
\[{I_n} = ( - (n + \lambda_0/n){\alpha ^{n - 1}}, - n{\alpha ^{n - 1}}), n\ge 1.\]
Combining Lemma \ref{sign} with Proposition \ref{langley}, we obtain that $f(x)$ has precisely one root in the interval $I_n$ for large $n$. So the root can be written as $ - (n + {\theta _n}(\alpha )){\alpha ^{n - 1}}$, where ${\theta _n}(\alpha )=O(n^{-1})$, and \[\mathop {\lim }\limits_{n \to \infty } n{\theta _n} = \mathop {\lim }\limits_{n \to \infty } \frac{{{H_{n}}(\alpha )}}
{{{h_{n}}(\alpha )}} = \frac{{H(\alpha )}}
{{h(\alpha )}} = \sum\limits_{k = 1}^\infty  {\frac{k}
{{{\alpha ^k} - 1}}}.\]
Since \[\frac{{ - (n + 1 + {\theta _{n + 1}}){\alpha ^n}}}
{{ - (n + {\theta _n}){\alpha ^n}}} = 1 + \frac{1}
{n} + o({n^{ - 2}}),\]
by Proposition \ref{langley} we immediately derive that the root in the interval $I_n$ is just $x_n$ for large $n$, namely ${x_n} =  - (n + {\theta _n}){\alpha ^{n - 1}}$.
Noting that \[\sum\limits_{k = 1}^\infty  {\frac{k}
{{{\alpha ^k} - 1}}}  = \sum\limits_{k = 1}^\infty  {\sigma (k)} {\alpha ^{ - k}},\]
and $1/q=\alpha$, we complete the proof.\end{proof}

\begin{proposition}\label{morris}The function $f(x)$ in (\ref{fx}) can be factored into an infinite product
\[f(x) = \prod\limits_{n = 1}^\infty  {\Big( {1 - \frac{x}
{{{x_n}}}} \Big)} ,\]
where each $x_n$ is the zero of $f(x)$ and $x_{n+1}<x_n<0$.
\end{proposition}
\begin{proof}See Morris et al.\cite{morris1972phragmen}.
\end{proof}

Since $f'(x)=f(x/\alpha)$, the critical points of $f(x)$ are $\alpha x_n(n\ge 1)$. Note that each zero of $f(x)$ is simple and $x_{n+1}<\alpha x_n< x_n, n\ge 1$, then the signs of the extrema $f(\alpha x_n)$ are alternate, namely ${( - 1)^n}f(\alpha {x_n}) > 0$. Then we can obtain an asymptotic formula for the oscillation amplitude of $f(x)$ on the negative real axis by the infinite product representation above.

\begin{proof}[Proof of Corollary \ref{coro}] By Theorem \ref{thm}, we have
\[1 - \frac{{\alpha {x_n}}}
{{{x_{n + 1}}}} = \frac{1}
{{n + 1}} + o({n^{ - 2}}),\]
\[\left| {\frac{{\alpha {x_n}}}
{{{x_k}}}} \right| \ge {\alpha ^{n + 1 - k}},\ k< n+1 \quad {\rm and} \quad \left| {\frac{{\alpha {x_n}}}
{{{x_k}}}} \right| < {\alpha ^{n + 1 - k}},\  k\ge n+1.\]

Let $\varepsilon_n=(1+H(\alpha)/h(\alpha))/n$. Since for large $n$, $ |x_n|=(n+\theta_n)\alpha^{n-1}, \theta_n=O(n^{-1})$,
by the Stirling's approximation for Gamma function, we have
\[\begin{gathered}
  \left| {f(\alpha {x_n})} \right| = \left( {1 - \frac{{\alpha {x_n}}}
{{{x_{n + 1}}}}} \right)\prod\limits_{k = 1}^n {\left( {\frac{{\alpha {x_n}}}
{{{x_k}}} - 1} \right)} \prod\limits_{k = n + 2}^\infty  {\left( {1 - \frac{{\alpha {x_n}}}
{{{x_k}}}} \right)}  \hfill \\
   > \frac{1}
{{2n}}\frac{{|\alpha {x_n}{|^n}}}
{{|{x_1} \cdot  \cdot  \cdot {x_n}|}}\prod\limits_{k = 1}^n {\left( {1 - \left| {\frac{{{x_k}}}
{{\alpha {x_n}}}} \right|} \right)} \prod\limits_{k = 1}^\infty  {\left( {1 - {\alpha ^{ - k}}} \right)}  \hfill \\
   > \frac{1}
{{2n}}\frac{{{n^n}{\alpha ^{{n^2}}}}}
{{{C_0}\Gamma(n + {\varepsilon_n+1}){\alpha ^{n(n - 1)/2}}}}\phi {(\alpha )^2} > {C_1}{n^{ - 3/2}}{e^n}{\alpha ^{n(n + 1)/2}} \hfill \\
\end{gathered} \]
where $\phi (\alpha ) = \prod\limits_{k = 1}^\infty  {(1 - {\alpha ^{ - k}})}$, $C_0(\alpha)$ and $C_1(\alpha)$ are positive numbers independent of $n$.

On the other hand, since $f(x)$ is entire, by Proposition \ref{morris} we have\[\sum\limits_{{{i_1} <  \cdot  \cdot  \cdot  < {i_n}}} {\frac{1}
{{{x_{{i_1}}}{x_{{i_2}}}\cdot\cdot\cdot{x_{{i_n}}}}}}  = \frac{{{{( - 1)}^n}}}
{{n!}}{\alpha ^{ - n(n - 1)/2}}.\]
Hence we get
\[\begin{gathered}
  \left| {f(\alpha {x_n})} \right| = \left( {1 - \frac{{\alpha {x_n}}}
{{{x_{n + 1}}}}} \right)\prod\limits_{k = 1}^n {\left( {\left| {\frac{{\alpha {x_n}}}
{{{x_k}}}} \right| - 1} \right)} \prod\limits_{k = n + 2}^\infty  {\left( {1 - \left| {\frac{{\alpha {x_n}}}
{{{x_k}}}} \right|} \right)}  \hfill \\
   < \frac{2}
{{n + 1}}\frac{{|\alpha {x_n}{|^n}}}
{{|{x_1}\cdot\cdot\cdot{x_n}|}}\prod\limits_{k = 1}^n {\left( {1 + \left| {\frac{{{x_k}}}
{{\alpha {x_n}}}} \right|} \right)} \prod\limits_{k = 1}^\infty  {\left( {1 + {\alpha ^{ - k}}} \right)}  \hfill \\
   < \frac{2}
{{n{\text{  +  }}1}}\frac{{{{(n + 1)}^n}{\alpha ^{{n^2}}}}}
{{n!{\alpha ^{n(n - 1)/2}}}}\Phi {(\alpha )^2} < {C_2}{n^{ - 3/2}}{e^n}{\alpha ^{n(n + 1)/2}} \hfill \\
\end{gathered} \]
where $\Phi (\alpha ) = \prod\limits_{k = 1}^\infty  {(1 + {\alpha ^{ - k}})}$ and $C_2(\alpha)$ are positive numbers independent of $n$. Then the proof is complete.
\end{proof}

\section*{Acknowledgments}
I am grateful to Professor Dexing Kong and Professor Wei Luo for all their help with this paper. I am also grateful to Professor Alan Sokal for going through an early draft of this paper and his helpful suggestions. It is my pleasure to thank my former colleagues Jianxin Sun and Liuquan Wang for their helpful comments.
\bibliographystyle{amsplain}
\bibliography{xbib}

\end{document}